\documentclass[11pt,reqno]{amsart}
\usepackage{amssymb,amsmath}\newcommand{\be}{\begin{eqnarray}}
\newcommand{\ee}{\end{eqnarray}}\newcommand{\e}{{\varepsilon}}\newcommand{\Th}{{\theta}}\newcommand{\R}{{\mathbb R}}\newcommand{\Nat}{{\mathbb N}}\newcommand{\Hau}{{\mathcal H}}\newcommand{\Cant}{{\mathcal C}}\newcommand{\K}{{\mathcal K}}\newcommand{\Rk}{{\mathcal R}}\newcommand{\ccc}{a}
\newcommand{\Proj}{\operatorname{Proj}}\newcommand{\Fav}{\operatorname{Fav}}\newtheorem{theorem}{Theorem}\theoremstyle{definition}\theoremstyle{remark}\numberwithin{equation}{section}\input epsf.sty
\begin{document}\thispagestyle{empty}

\title[Buffon needle probability of the four-corner Cantor set]{{An estimate from below for the Buffon needle probability of the four-corner Cantor set}}

\author{Michael Bateman}\address{Michael Bateman, Department of Mathematics, Indiana University
\newline{\tt mdbatema@indiana.edu}}
\author{Alexander Volberg}\address{Alexander Volberg, Department of  Mathematics, Michigan State University and the University of Edinburgh
{\tt volberg@math.msu.edu}\,\,and\,\,{\tt a.volberg@ed.ac.uk}}

\thanks{Research of the authors was supported in part by NSF grants  DMS-0501067 (Volberg) and DMS-0653763 (Bateman)}
\begin{abstract}Let $\Cant_n$ be the $n$-th generation in  the construction of the middle-half Cantor set. The Cartesian square $\K_n =  \Cant_n \times \Cant _n$ consists of $4^n$
squares of side-length $4^{-n}$. The chance that a long needle
thrown at random in the unit square will meet $\K_n$ is
essentially the average length of the projections of $\K_n$, also
known as the Favard length of $\K_n$. A classical theorem of
Besicovitch implies that the Favard length of  $\K_n$ tends to
zero. It is still an open problem to determine its exact rate of
decay. Until recently, the only explicit upper bound was $\exp(-
c\log_* n)$, due to Peres and Solomyak.  ($\log_* n$ is the number
of times one needs to take log to obtain a number less than $1$
starting from $n$). In \cite{NPV} the power estimate from above
was obtained. The exponent in \cite{NPV} was less than $1/6$ but
could have been slightly improved. On the other hand, a simple
estimate shows that from below we have the estimate $\frac{c}{n}$.
Here we apply the idea from \cite{katz}, \cite{BK} to show that
the estimate from below can be in fact improved to $c\,\frac{\log
n}{n}$. This is in drastic difference from the case of {\em
random} Cantor sets studied in \cite{PS}. \end{abstract}\maketitle

\section{{\bf Introduction}} \label{sec:intro}
\begin{figure}[htbp]\centerline{\epsfxsize=2.in \epsffile{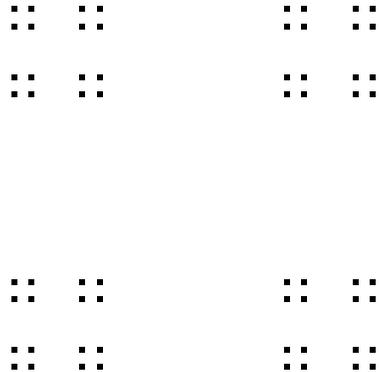}} \caption{$\K_3$, the third stage of the construction of $\K$.}\end{figure}
The four-corner Cantor set $\K$ is constructed by replacing the unit square by four sub-squares of side length $1/4$ at its corners, and
iterating this operation in a self-similar manner in each sub-square.  More formally,
consider  the set  $\Cant_n$ that is the union of $2^n$ segments:

\[
\Cant_n = \bigcup_{\ a_j \in \lbrace 0,3  \rbrace, j=1,..,n}\Bigl[\sum_{j=1}^n a_j 4^{-j}, \sum_{j=1}^n a_j 4^{-j} +4^{-n}\Bigr]\,,
\]
and  let the middle half Cantor set be
\[
\Cant :=\bigcap_{n=1}^{\infty} \Cant_n\,.
\]

 (It can also be written as
$\Cant = \lbrace \sum_{n=1}^\infty a_n 4^{-n}:\ a_n \in \lbrace 0,3 \rbrace \rbrace.$)  The four corner Cantor set $\K$ is the Cartesian square $\Cant \times \Cant$.

Since the one-dimensional Hausdorff measure of $\K$ satisfies
$0<\Hau^1(\K)<\infty$ and the projections of $\K$ in two distinct
directions have zero length, a theorem of Besicovitch (see
\cite[Theorem 6.13]{falc1})  yields that the projection of $\K$ to
almost every line through the origin has zero length. This is
equivalent to saying that the Favard length of $\K$ equals zero.
Recall (see \cite[p. 357]{besi}) that the {\bf Favard length} of a
planar set $E$ is defined by\begin{equation} \label{favdef}\Fav(E)
= \frac{1}{\pi} \int_0^\pi |\Proj \Rk_\Th
E|\,d\Th,\end{equation}where $\Proj$ denotes the orthogonal
projection from $\R^2$ to the horizontal axis,  $\Rk_\Th$ is the
counterclockwise rotation by angle $\Th$, and $|A|$ denotes the
Lebesgue measure of a measurable set $A \subset \R$. The Favard
length of a set $E$ in the unit square has a probabilistic
interpretation: up to a constant factor, it is the probability
that the ``Buffon's needle,'' a long line segment dropped at
random, hits $E$  (more precisely, suppose the needle's length is
infinite, pick its direction uniformly at random, and then locate
the needle in a uniformly chosen position in that direction, at
distance at most $\sqrt{2}$ from the center of the unit square).

The set $\K_n=\Cant_n^2$  is  a union of $4^n$ squares with side length $4^{-n}$ (see Figure 1 for a picture of $\K_3$). By the dominated convergence theorem, $\Fav(\K)=0$ implies  $\lim_{n\to\infty} \Fav(\K_n)=0$. We are interested in  good estimates for $\Fav(\K_n)$ as $n\to \infty$. A lower bound $\Fav(\K_n)\ge \frac{c}{n}$ for some $c>0$ follows from  Mattila \cite[1.4]{mattila1}.   Peres and Solomyak \cite{PS} proved that

\[
\Fav(\K_n) \le C \exp[-\ccc\log_*n]\ \ \ \ \mbox{for all}\ \ n \in \Nat,\] where  \[\label{def-logstar}\log_* n = \min\left\lbrace k\ge 0:\ \underbrace{\log\log\ldots \log}_{k}n \le 1 \right \rbrace\,.
\]

This result can be viewed as an attempt to make  a quantitative statement out of a qualitative Besicovitch projection theorem \cite{besi}, \cite{T}, using this canonical example of the Besicovitch irregular set.

It is very interesting to consider quantitative analogs of Besicovitch theorem in general. The reader can find more of that in \cite{T}.

In \cite{NPV} the following estimate from above was obtained
\[
\Fav(\K_n) \le \frac{C_{\tau}}{n^{\tau}}\,,
\]
where $\tau$ was strictly less than $1/6$. This can be slightly improved, but it is still a long way till $\tau =1$. Here we show, using the idea of \cite{BK}, \cite{katz}, that $\tau=1$ is impossible.

\begin{theorem}
\label{th-corner}
There exists $c>0$ such that
\begin{equation}
\label{log}
\Fav(\K_n) \ge c\,\frac{\log n}{n}\ \ \ \ \mbox{for all}\ \ n\in \Nat.
\end{equation}
\end{theorem}

\noindent{\bf Remark.}  This result is somewhat surprising in light of the probabilistic result in \cite{PS}.  There, the authors consider planar  Cantor sets constructed randomly as follows.  Starting from the unit square $U$, divide $U$ into four equal squares $U_1, U_2, U_3, U_4$.  Similarly divide each of these into four squares $U_{j1}, U_{j2}, U_{j3}, U_{j4}$.  For each $j$, randomly choose one square $U_{jk}$ (of side length $ {1 \over {16} }$ ).  The four chosen squares form the first level $ \tilde{ \K } _1 $.  Repeat this process, always choosing the next generation randomly.  The authors in \cite{PS} show that one expects
\[   { 1 \over  {Cn} } \leq \Fav( \tilde { \K_n }  )  \leq { C \over n} .\]

\begin{proof}

The proof is an immediate corollary of the idea of \cite{katz} if one applies the duality between Cantor sets and Kakeya sets from \cite{mattila2}.
As a ``warm-up" we are going to prove a much simpler estimate
\begin{equation}
\label{simpler}
\Fav(\K_n) \ge \frac{c}{n}\ \ \ \ \mbox{for all}\ \ n\in \Nat\,.
\end{equation}
This does not require \cite{BK}, \cite{katz}.

In what follows the square means only the Cantor square.  Let $L_{\theta}$ be the line passing through the origin at an angle $\theta$ with the $x$-axis.
Let $f_{n, \theta}(x)$ denote the number of squares in $\K _n$ whose orthogonal projection onto the line $L_{\theta}$ contains a point $x$ of this line.
For each square $Q$ of size $4^{-n}$ let $\chi_{Q, \theta}(x) $ be the characteristic function of the projection of $Q$ onto $L_{\theta}$.
Then $f_{n, \theta}(x) =\sum_{Q, \ell(Q)=4^{-n}} \chi_{Q, \theta}(x) $ .
Therefore,
\begin{equation}
\label{L1}
\int \int f_{n, \theta}(x) dx\,d\theta \asymp 4^n\cdot 4^{-n} =1\,.
\end{equation}

Let us denote the support of $f_{n, \theta}(x)$  by $E_{n, \theta}$,   $|E_{n, \theta}|$ being its length.

Knowing the first and second  moment of $f_{n, \theta}(x)$ we can estimate
$\int |E_{n, \theta}|\,d\theta$ by using Cauchy inequality twice:
$$
1\asymp \int\int f_{n,\theta} dx d\theta \leq \int |E_{n,\theta}|^{\frac12} (\int f_{n,\theta}^2(x) dx)^{\frac12}\, d\theta \le
$$
$$
(\int |E_{n,\theta}|\,d\theta)^{\frac12}  (\int \int f_{n,\theta}^2(x) dx\,d\theta)^{\frac12}\,.
$$

Hence,
\begin{equation}
\label{sn}
\int |E_{n,\theta}|\,d\theta \ge c\, \frac{1}{\int\int  f_{n,\theta}^2(x) dx\,d\theta}\,.
\end{equation}
Now
$$
\int\int  f_{n,\theta}^2(x) dx\,d\theta = \sum_{Q, Q',  \ell(Q)=\ell(Q')=4^{-n}}\int \int\chi_{Q, \theta}(x)  \chi_{Q', \theta}(x) dx\,d\theta\,.
$$
So for each pair $P=(Q, Q'),  \,\,\ell(Q)=\ell(Q')=4^{-n}$ ($Q$ and $Q'$ may coincide) we consider
\begin{equation}
\label{pP}
p_P := \int |Proj_{\theta}Q\cap Proj_{\theta}Q'|\,d\theta\,.
\end{equation}

Let us make an order on pairs. We call a pair $P$ a $k$--pair if $ Q, Q'$ are in a $4^{-k}$--square, but not in any $4^{-k-1}$--square, $k=0,1,...,n$.
We have $4^k$ of $4^{-k}$--squares, so we have $\asymp 4^k\cdot (4^{n-k})^2$ $k$--pairs. For each $k$--pair $P$ we obviously have
$$
p_P \le C\, 4^{-n} 4^{k-n}\,.
$$
Putting this together we get
$$
\int\int  f_{n,\theta}^2(x) dx\,d\theta = \sum_{P}p_P =\sum_{k=0}^{n} \sum_{P\,\text{is a}\, k-\text{pair}}p_P
$$
$$
 \le C\,\sum_{k=0}^{n-1} \sum_{P\,\text{is a}\, k-\text{pair}}4^{-n} 4^{k-n}\le
C\, \sum_{k=0}^{n-1}4^k\cdot (4^{n-k})^24^{-n} 4^{k-n}\le C\,n\,.
$$
This estimate and \eqref{sn} give us
$$
\int |E_{n,\theta}|\,d\theta \ge \frac{c}{n}\,.
$$

\bigskip

To prove \eqref{log} one needs to count pairs in a much more
interesting way, which one gets from \cite{BK}.

First we consider axis $0X$, where $0$ is the origin and the axis has angle $\arctan\frac12$ with the horizontal axis. We also need $0Y$, the orthogonal axis.
Project original unit square on $0X$. We obtain the segment
 $I_0:=[0,L], L=\sqrt{2}\cos (\frac{\pi}{4}-\arctan\frac12)$ on
 $0X$.
Notice that projections of Cantor squares of size $4^{-k}$, $k=0,...,n$, generate the $4$-adic structure on $I_0=[0,L]$. Segments of these
$4$-adic structure will be called $I_{\sigma}$, where $\sigma$ is the word of length at most $n$ in the alphabet of $\{0,1,2,3\}$.

We have $4^n$ points that are the projections of the centers of $4^n$ squares  $Q$ of size $4^{-n}$.
 We will call this set  $S$, and use the notation $s$
(maybe with indices) for elements  of $S$. Each $s$ recovers its $Q_s$ uniquely. Let $y_s$ be the $0Y$ coordinate of the center of $Q_s$.  Note that each $s$ is the center of an interval $I_{\sigma}$, and that the projections of all cubes $Q$ onto this axis are disjoint.  This is an important feature of the argument.

Along with the usual Euclidean distance $|s_1-s_2|$ between the points $s_1, s_2\in S$,
we have another very simple distance which will play {\em the crucial role} in proving  \eqref{log}. Namely,
$$
d(s_1,s_2) := \min \{ | I_{\sigma}|, \,s_1\in I_{\sigma} \,s_2\in I_{\sigma}\} \,.
$$
This is just the usual $4$-adic distance scaled by $L$.
Of course $|s_1-s_2|\le d(s_1, s_2)$.

\bigskip

For $j=0,1,..,\log n$, $k\in [-n+ j, 0],$ we call pair $P$ a $(j,k)$-pair , if
$$
\frac{|s_1-s_2|}{|y_{s_1}-y_{s_2}|}\asymp 4^{-j}\,,
|s_1-s_2| \asymp 4^{-k-j}\,.
$$

Now the pair $P=(Q,Q')$ of squares of size $4^{-n}$ is just a pair $(s_1,s_2),s_i\in S$.

For every $(j,k)$-pair $P=(s_1,s_2)$ one immediately has
\begin{equation}
\label{pP1}
p_P \le C\, \frac{1}{4^n} \cdot \frac{4^{-n}}{ |y_{s_1} - y_{s_2}|}\,.
\end{equation}
where $p_P$ is as in (\ref{pP} ).
Now we want to estimate the number $A_{j,k}$ of all $(j,k)$-pairs. If $(s_1,s_2)$ is a $(j,k)$-pair, then
$$
|s_1-s_2| \asymp 4^{-k-j}\,
$$ But also
$$
4^j |s_1-s_2|\le C\, |y_{s_1}-y_{s_2}| \,,
$$
and
\begin{equation}
\label{crucial}
 |y_{s_1}-y_{s_2}| \le C' d(s_1,s_2)\,.
 \end{equation}
The last inequality is {\em obvious} but it is the most crucial for the proof!

This is because  we just obtained  $d(s_1,s_2) \ge c 4^{-k}$. How many $4$-adic intervals are such that  $d(s_1,s_2) \ge c 4^{-k}\ge 4^{-k-a}$  ($a$ is absolute), and $|s_1-s_2| \asymp 4^{-k-j}$?
Corresponding two $4$-adic  intervals of size $4^{-n}$ should be both in $C\,4^{-k-j}$-neighborhood of the $4$-adic points of $1, 2, 3,..., k, k+1, .., k+a$-generations.
We have $4, 4^2,..., 4^{k+a}$ such points correspondingly.

Therefore,
$$
A_{j,k} \le C\, \sum_{m=0}^{k+a} 4^m (\frac{4^{-k-j}}{4^{-n}})^2 = C\,4^{2n-k-2j}\,
$$

Another way to count the number of $(j,k)$ pairs is as follows.

Note that if $j=0$, there would be $4^{2(n-k) }4^k$ such pairs, since there are $4^k$ intervals of
length $4^{-k}$, and each contains $4^{2(n-k)}$ pairs of intervals of length $4^{-n}$.
 Increasing $j$ by $1$ decreases the number of pairs by a factor of ${ 1 \over {4^2} }$.
 One can see this by noting that if a pair $s_1$, $s_2$, satisfies these conditions,
 then the $4$-adic expansions of $s_1$ and $s_2$ are almost uniquely determined for $j$ digits.
 Hence

$$
A_{j,k} \leq C 4^{2n-k - 2j}.
$$
Using this and \eqref{pP1} we get
$$
\sum_{p\in (j,k)-\text{pairs}} p_P \le C\, 4^{2n-k-2j} \frac{4^{-2n}}{4^{-k}} \asymp 4^{-2j}\,.
$$
The union of all  $(j,k)$-pairs over all $k$ is called: $\mathcal{P}_j'$.

So fix $j$, and get
\begin{equation}
\label{j}
\sum_{p\in \mathcal{P}_j'} p_P =\sum_{k=-n+ j}^{0}\sum_{p\in (j,k)-\text{pairs}} p_P\le C\, \frac{n}{4^{2j}}\,.
\end{equation}

\bigskip

Now let $J_j := [c_1 4^{-j}, c_2 4^{-j}]$, where $c_1$ is sufficiently small and $c_2$is sufficiently large.
These are intervals of {\em angles} $\theta$ with respect to the axis $0X$, where zero angle means we are line parallel to the axis $0X$.

\bigskip

Here is a crucial geometric observation:

\begin{multline}
\label{crucialgeometric}
\text{If}\,\,\,P=(Q,Q'), Q\neq Q' \,\,\,\text{ is so that}\,\,\,
\\Proj_{\theta} Q \cap Proj_{\theta} Q' \neq \emptyset\,\,,
\theta\in J_j\,\,\,\text{ then }\,\,\,\,P\in \mathcal{P}_j'\,.
\end{multline}

Let us throw into $\mathcal{P}_j' $ also all $(Q,Q)$ pairs. The
resulting collection is called $j$-pairs: $\mathcal{P}_j$. As
$$
\int_{J_j} |E_{n,\theta}|\, d\theta \ge c\, \frac{(\int_{J_j}\int f_{n,\theta}dx\,d\theta)^2}
{\int_{J_j}\int f^2_{n,\theta}dx\,d\theta}\,,
$$
\begin{multline}
\label{Pj}
\int_{J_j}\int f^2_{n,\theta} dx\,d\theta\le \sum_{p\in \mathcal{P}_j} p_P\le C\, \frac{n}{4^{2j}}+\\
\int_{J_j} \int\sum_{P=(Q,Q), \ell(Q)=4^{-n}}\chi_{Q,\theta}(x)dx\,d\theta \le C\, \frac{n}{4^{2j}}+C\,4^{-j}4^n4^{-n}\le C\, \frac{n}{4^{2j}}\,,
\end{multline}
and
$$
\int_{J_j}\int f_{n,\theta}dx\,d\theta \le C\, |J_j|\cdot 4^n \cdot 4^{-n} \asymp 4^{-j}\,,
$$
we combine this to obtain
\begin{equation}
\label{Jj}
\int_{J_j} |E_{n,\theta}|\, d\theta \ge c\, 4^{-2j}\frac{4^{2j}}{n} =\frac{c}{n}\,.
\end{equation}

\noindent{\bf Remark.} Notice that if \eqref{Pj} stops to be valid if $j> \log_4 n + Const$.  This explains why we did not get a better estimate from below than that in the Theorem.

Summing \eqref{Jj} over $j=0,...,\log n$ we obtain \eqref{log}. Theorem is completely proved.

\end{proof}

\bigskip

\section{Median value of $|E_{n, \theta}|$}
\label{median}

\noindent{\bf Question 1}. What is the median value of   $|E_{n, \theta}|$?

Let us call this median value $M_n$. We can prove the following simple theorem, which immediately implies \eqref{simpler} of course.

\begin{theorem}
\label{medi}
$M_n \ge \frac{c}{n}\,.$
\end{theorem}

\begin{proof}
We are going to prove
\begin{equation}
\label{over}
\int \frac{1}{ |E_{n, \theta}|} \, d\theta \le C\, n\,.
\end{equation}
 If one uses Tchebyshev's inequality
this immediately gives $M_n\ge \frac{c}{n}$.

To prove  \eqref{over} we use \cite{mattila2}.  Let us fix a small positive $\e$, and let $\mu_n$ be an equidistributed measure on $\Cant_n$. Let $Proj_{\theta}$ stand (as always) for the orthogonal projection onto line $L_{\theta}$.
Notice that given two points $z, \zeta\in \mathbb{C}$ we have
$$
\frac{\e}{|z-\zeta|} \asymp |\{\theta : |Proj_{\theta}(z) - Proj_{\theta}(\zeta)| \le \e\||\,.
$$
Using this we write
$$
\int\int \frac{\e}{|z-\zeta|} d\mu_n(z) d\mu_n(\zeta) \asymp \int\int  |\{\theta : |Proj_{\theta}(z) - Proj_{\theta}(\zeta)|\le\e\}| d\mu_n(z) d\mu_n(\zeta)
$$
Introduce
$$
\Phi_{\e}(x)=\begin{cases} 1\,,\,\text{if}\,\, 0\le x\le \e\\
0\,,\,\,\text{otherwise}\end{cases}
$$
Then we repeat
$$
\int\int \frac{\e}{|z-\zeta|} d\mu_n(z) d\mu_n(\zeta) \asymp\int\,\int\int \Phi_{\e}(|Proj_{\theta}(z) - Proj_{\theta}(\zeta)|)d\mu_n(z) d\mu_n(\zeta)  d\theta  =
$$
$$
\int\,\int\int \Phi_{\e}(|x - y|)d\mu_{n,\theta}(x) d\mu_{n,\theta}(y)  d\theta \,,
$$
where $d\mu_{n,\theta}$ is the projection of the measure $\mu_n$ on the line $L_{\theta}$. In our old notation
\begin{equation}
\label{density}
d\mu_{n,\theta} = f_{n,\theta}(x)\,dx\,.
\end{equation}
Of course
$$
\int \Phi_{\e}(|x - y|)d\mu_{n,\theta}(y) = \mu_{n,\theta}(B(x,\e))\,,
$$
and
finally we get
\begin{equation}
\label{energy}
\int\int \frac{1}{|z-\zeta|} d\mu_n(z) d\mu_n(\zeta)\ge c\,\int \int \frac{ \mu_{n,\theta}(B(x,\e))}{\e} \,d\mu_{n,\theta}(x)\,d\theta\,.
\end{equation}
The left hand side is $\le C\,n$. One can see this by noting that for each square $Q$ of side length $4^{-n}$ in $\K _n$, and for each  $k=0,1,...,n$,  there are $4^{n-k}$ squares $Q`$ at distance $4^{-k}$.

In \eqref{energy} we now use Fatou's lemma:
\begin{equation}
\label{energy1}
\int \liminf_{\e\rightarrow 0} \frac{ \mu_{n,\theta}(B(x,\e))}{\e} \,d\mu_{n,\theta}(x)\,d\theta\le C\, n\,.
\end{equation}
Recalling \eqref{density} we obtain
\begin{equation}
\label{energy2}
\int \int_{E_{n, \theta}}f_{n,\theta}(x)^2 \,d x\,d\theta\le C\, n\,.
\end{equation}
Recalling \eqref{L1} we can rewrite it as
\begin{equation}
\label{energy3}
\int\frac{ \int_{E_{n, \theta}}f_{n,\theta}(x)^2 \,d x}{(\int_{E_{n, \theta}}f_{n,\theta}(x)\,dx)^2} \,d\theta\le C\, n\,.
\end{equation}
By Cauchy inequality
$$
\frac{1}{|E_{n,\theta}|} \le \frac{ \int_{E_{n, \theta}}f_{n,\theta}(x)^2 \,d x}{(\int_{E_{n, \theta}}f_{n,\theta}(x)\,dx)^2}\,.
$$
Combine this and \eqref{energy3} and obtain the desired estimate
$$
\int \frac{1}{|E_{n,\theta}|} \,d\theta \le C\,n\,.
$$
Inequality \eqref{over} and, therefore,  Theorem \ref{medi} are completely proved.

\end{proof}

\section{Sierpi\'nski's Cantor set}
\label{Sierp}

Consider now another Cantor set, which, by analogy with Sierpi\'nski's gasket, we call Sierpi\'nski's Cantor set $\mathcal{S}$.
We take an equilateral triangle with side lenghth $1$, leave $3$ triangles of size $1/3$ at each corner, and then continue this for $n$ generations. On step $n$
we get $3^n$ equilateral tringles of size $3^{-n}$. Call this union of triangles $\mathcal{S}_n$. Its  intersection is $\mathcal{S}$,
$$
0<H^1(\mathcal{S}) <\infty\,,
$$
and this is a Besicovitch irregular set, so, by Besicovitch projection theorem (see \cite{mattila3})
$$
\zeta_n:=\int |\mathcal{S}_{n, \theta}|\,d\theta\rightarrow 0,\,\, n\rightarrow \infty\,.
$$

 \noindent{\bf Question 2}. What is the order of magnitude of $\zeta_n$?

 \bigskip

 This is the same question, which we had for $4$-corner Cantor set.

 Absolutely the same reasoning as above proves

 \begin{theorem}
 \label{Sier}
\[ \zeta_n = \int |\mathcal{S}_{n, \theta}|\,d\theta \ge c\frac{\log n}{n}\,.\]
 \end{theorem}

 In fact, projection of the triangles on the base side generate $3$-adic lattice on the base side. Then we notice that \eqref{crucial} and \eqref{crucialgeometric} hold now as well.
 The proof is the same after these observations.

 \bigskip


  \bibliographystyle{amsplain}

\begin{thebibliography}{99}
  \bibitem{BK} M. Bateman, N.Katz, {\em Kakeya sets in Cantor directions},  arXiv:math/0609187v1, 2006, pp. 1--10.
  \bibitem{besi} A. S. Besicovitch, {\em Tangential properties of sets and arcs of infinite linear measure}, {\em Bull.\ Amer.\ Math.\ Soc.}\ {\bf 66} (1960), 353--359.
  \bibitem{falc1} K. J. Falconer,  The geometry of fractal sets.
  Cambridge Tracts in Mathematics, 85. C.U.P., Cambridge--New York, (1986).
\bibitem{katz} Katz, N. H.,  {\em A counterexample for maximal operators over a Cantor set of directions} Math. Res. Let. 3 (1996), pp. 527-–536.
 \bibitem{keich} U. Keich, {\em On $L^p$ bounds for Kakeya maximal functions and the  Minkowski dimension in $\R^2$}, {\em Bull. London. Math. Soc.} {\bf 31} (1999),  pp. 213--221.
  \bibitem{kenyon} R. Kenyon,  {\em Projecting the one-dimensional Sierpinski gasket, } {\em Israel J. Math.} {\bf 97} (1997), 221--238.
  \bibitem{lawang} J. C. Lagarias and Y. Wang, {\em Tiling the line with translates of one tile}, {\em Invent.\ Math.}{\bf 124} (1996), 341--365.
  \bibitem{mattila1} P. Mattila, {\em Orthogonal projections, Riesz capacities and Minkowski content}, {\em Indiana Univ.\ Math.\ J.}\ {\bf 39} (1990), 185--198.
   \bibitem{mattila2} P. Mattila, {\em Hausdorff dimension, projections, and the Fourier tarnsform}, Publ. Mat.,  {\bf 48} (2004), pp. 3--48.
  \bibitem{mattila3} P. Mattila, {\em Geometry of Sets and Measures in Euclidean Spaces}, Cambridge University Press, 1995.
 \bibitem{NPV} F. Nazarov, Y. Peres, A. Volberg, {\em  The power law for the Buffon needle probability of the four-corner Cantor set}, arXiv:0801.2942, 2008, pp. 1--15.
  \bibitem{pesiso} Y. Peres, K. Simon and B. Solomyak, {\em Self-similar sets of zero Hausdorff measure and positive packing measure}, {\em Israel J.\ Math.} {\bf 117} (2000),353--379.
  \bibitem{PS} Y. Peres and B. Solomyak, {\em How likely is {B}uffon's needle to fall near a planar {C}antor set?} \newblock {\em Pacific J. Math. 204}, 2 (2002), 473--496.
  \bibitem{schoenberg} I. J. Schoenberg, {\em On the Besicovitch-Perron solution of the Kakeya problem}, {\em Studies in mathmatical analysis and related topics}, 
  \bibitem{T} T. Tao, {\em A quantitative version of the Besicovitch projection theorem via multiscale analysis,} pp. 1--28, arXiv:0706.2446v1 [math.CA] 18 Jun 2007.
  \end{thebibliography}

\end{document}